\documentclass[a4paper, 10pt, draft]{ieeeconf}      

\IEEEoverridecommandlockouts                              
\title{\LARGE \bf
An Application of Joint Spectral Radius in Power Control Problem for Wireless Communications*}

\author{Vahid S. Bokharaie$^{1}$ and Gholamreza Parsaee$^{2}$
\thanks{*This work was supported by Science Foundation Ireland award SFI/09/SRC/E1780.}
\thanks{$^{1}$V. S. Bokharaie is with Mathematics Applications Consortium for Science and Industry (MACSI), Univeristy of Limerick, Ireland.
        {\tt\small vahid.bokharaie@ul.ie}}%
\thanks{$^{2}$G. Parsaee is with Plano-net Engineering Co.
        {\tt\small reza\_parsaee@yahoo.com}}%
}
\usepackage{amsmath,amssymb}			

\newtheorem{thm}{Theorem}[section]

\newtheorem{definition}{Definition}[section]

\newcommand{\Rp}{\mathbb{R}^n_+}%

\newcommand{\NN}{\mathbb{N}}
\newcommand{\RR}{\mathbb{R}}                                  

\newcommand{\BB}{\mathcal{B}}                                  
\newcommand{\MM}{\mathcal{M}}

\begin{document}

\maketitle
\thispagestyle{empty}
\pagestyle{empty}



\begin{abstract}                          
Resource management, including power control, is one of the most essential functionalities of any wireless telecommunication system. Various transmitter power-control methods have been developed to deliver a desired quality of service in wireless networks. We consider two of these methods: Distributed Power Control and Distributed Balancing Algorithm schemes. We use the concept of joint spectral radius to come up with conditions for convergence of the transmitted power in these two schemes when the gains on all the communications links are assumed to vary at each time-step.
\end{abstract}

\section{Introduction} \label{sec:intro}
In this paper, we deal with the problem of power control in a mobile network. The main idea behind power control is to transmit the signal with minimum power required to maintain the desirable quality of service (QoS) \cite{Stu11}. The immediate benefit of power control is saving energy. It helps to prolong mobile battery life and also cut the cost of network operation on base station subsystem (BSS) \cite{CHLW08}. In majority of today's wireless networks, however, power control is used for a more important aim: to reduce interference in the common media, air. It therefore helps to obtain better quality under similar circumstances and an increase in network capacity \cite{CHLW08,CT91}. 
It can be implied from the above, that the main problem in power control is finding the optimum set of transmission powers for each pair of communicators. The answer depends on three main factors \cite{TZ07}: (i) the desired QoS, which is normally expressed in terms of demanded throughput, tolerable bit error rate and delay \cite{Peu99}, (ii) level of interference in the air, and (iii) characteristics of the channel between the certain transmitter and receiver in question, such as instant attenuation or fading profile. Desired QoS can be roughly translated to a minimum requirement for Signal to Noise plus Interference Ratio (SINR) \cite{ML02}. Knowing SINR requirement, the present level of interference and noise power, minimum needed signal power at receiver antenna can be calculated. Finally by adding the channel attenuation we can obtain the minimum power that the transmitter must radiate. Apart from the problem of getting an estimation of channel attenuation and interference at the transmitter side, the main challenge here is that interference level is itself affected by the behaviour of all other users, as reducing or increasing power for one user changes the interference observed by others, causes them to change their power in turn, and so on.

Addressing these issues, in this manuscript, we consider a mobile network in which the gain on all communication links vary in each time step. Using the concept of joint spectral radius (as defined in the next section), we present conditions for convergence of the transmitter power in Distributed Power Control (DPC) and Distributed Balancing Algorithm schemes (DBA) schemes.

\section{Problem Description} \label{sec:bground}
We consider a mobile network with $N$ channels (where a channel can be a frequency band or a time slot), and channels are being reused according to some arbitrary channel assignment scheme. We assume $m$ mobiles are using the channels and without loss of generality we only study the uplinks, but depending on the radio access systems in use the presented results might be applicable to downlinks as well.

We define

\begin{itemize}
\item {$P_i$}: Power transmitted by the $i$th mobile. 
\item {$G_{ij}$}: The gain on the communication link between the $i$th base and the $j$th mobile with $G_{ij}\geq 0$ for all $i,j=1, \cdots, m$.
\item{$\gamma_{i}$}: SINR at the base assigned to the $i$th mobile equipment.
\end{itemize}

Hence, $\gamma_{i}$ can be calculated as follows:
\[ \gamma_i=\frac{P_i G_{ii}}{\displaystyle \sum_{j\neq i}^m P_j G_{ij}}, \quad  i=1, \cdots, m\]

Now we consider the Distributed Power Control (DPC) scheme. In the DPC scheme, mobiles adjust their power synchronously in discrete time steps. The power adjustment by the $i$th mobile in the $n$th time instant is given by
\begin{equation}\label{eq:P_i_n}
P_i(n)=c(n-1) \frac{P_i(n-1)}{\gamma_i(n-1)}
\end{equation}
where $P_i(n)$ is the power transmitted by the $i$th mobile at the $n$th time step, $\gamma_i(n)$ is the resulting SINR and $c(n-1)\in \Rp$ is some constant.

Let $A$ be an $m\times m$ matrix defined as $A_{ij}=G_{ij}/G_{ii}$ for $i \neq j$ and $A_{ii}=0$ for $i=1, \cdots, m$. Thus, \eqref{eq:P_i_n} can be written as 
\begin{equation}\label{eq:PcAP}
P(n)=c(n-1)A(n-1)P(n-1)
\end{equation}
where
\[P(n):=
 \begin{pmatrix}
  P_1(n) &  \cdots &  P_{m-1}(n) & P_m(n) \end{pmatrix}^T\]
$P(n)$ can also be written as follows
\begin{equation}
P(n)=\Big(\prod_{i=0}^n c(i)\Big) \Big(\prod_{i=0}^n A(i)\Big) P(0)
\end{equation}
Note that the constants $c(i)$ for $i \in \NN$ doesn't affect SINR. It is a constant chosen properly to prevent the transmitter powers from becoming too large or too small.

For a matrix $M \in \RR^{n \times n}$, we denote the set of all eigenvalues of $M$ or \textit{spectrum} of $M$ by $\sigma(M)$. The \emph{spectral radius} of $M$ is represented by $\rho(M)$ and defined as follows:
\[\rho(M) := \max \{ |\lambda| : \lambda \in \sigma(M)\} \]
One property of spectral radius is that \cite[p. 14]{The05}:
\begin{equation}
\rho(M)=\lim_{k\rightarrow \infty}\|M^{k}\|^{1/k}
\end{equation}
where $\|\cdot\|$ represents any submultiplicative matrix norm. This property, has been the inspiration for Joint Spectral Radius which is introduced in \cite{RS60}. 
\begin{definition}
Associated to the set of matrices $\mathcal{M}$  is the
joint spectral radius, which is given by
\begin{equation}
\begin{split}
\rho(\mathcal{M}):= \lim_{t\to \infty} \sup \{ \| M(t-1)M(t-2)\dots
M(0)\|^{1/t} \\ \;|\; M(s) \in {\cal M}, s=0,\ldots,t-1 \} \,.
\end{split}
\end{equation}
\end{definition}
Although $\rho(\cdot)$ is used to represent both spectral radius of a matrix and joint spectral radius of a set of matrices, it should not lead to confusion. Also, it is easy to check that in the special case where the set $\MM$ contains only one matrix, joint spectral radius is equivalent to spectral radius.



In DPC Scheme, one aim is to make sure $P(n)$ will remain bounded as $n \rightarrow \infty$. We will use properties of the joint spectral radius to introduce a novel condition for convergence of the set of transmitter powers in DPC scheme. 
\begin{thm}\label{thm:DPC}
Consider a system that uses the DPC scheme, with $c(n)$, $n\in \NN$ chosen such that $\displaystyle \lim_{n\rightarrow \infty}\displaystyle\prod_{i=0}^n c(n)$ is bounded. Then $\displaystyle \lim_{n\rightarrow \infty}P(n)$ is bounded if $\rho(\BB)<1$ where
\[\BB=\{B(1), \cdots, B(n)\}\]
with $B(i)=c(i)A(i)$, for $i=1, \cdots, n$.
\end{thm}
\begin{proof}
Recall that $\rho(\mathcal{M})<1$ is equivalent to the existence of constants $C \geq 1,\gamma \in (0,1)$ such that for all $t\in \NN$ we have
\begin{equation}\label{eq:Cg}
\| M_{t-1} \cdots M_1 M_0 \| \leq C \gamma^{t} \quad 
\end{equation}
for all $M_s\in \mathcal{M}$, $s=0,\ldots,t$ \cite{Wir02}. 
Thus $\rho(\mathcal{M})<1$ characterizes exponential stability of the linear
inclusion
\begin{equation}\label{eq:inclu}
x(t+1) \in \left\{ M x(t) \;|\; M\in \mathcal{M} \right\}
\end{equation}
Since $\rho(\BB)<1$, we can apply this result to the linear inclusion \eqref{eq:PcAP} to conclude that $P(n)$ will remain bounded for all $n\in \NN$. This conludes the proof.
\end{proof}
Note that the elements of the matrix $A(i)$, for $i=1, \cdots, n$ is specified by the parameters of the system. On the other hand, we have control over constants $c(i)$, for $i=1, \cdots, n$ and we can use that to make sure the condition $\rho(\BB)<1$ is satisfied.

Also, note that if $\rho(\BB)>1$, we cannot conclude that $P(n)$ will be unbounded as $n\rightarrow \infty$. To further clarify this point, note that by the generalised spectral radius theorem \cite{BW92}\cite{Elsn95}, if $\rho(\BB)>1$, then there is a finite sequence $(B_0, B_1, \cdots, B_{t-1})$ such that the spectral radius of the product $\rho(B_{t-1}\cdots B_1B_0)>1$ and hence the corresponding linear inclusion \eqref{eq:P_i_n} is exponentially unstable. But it is possible that $\big(B(0), B(1), \cdots, B(n))$ might not be such a sequence. Therefore, the condition presented in Theorem \ref{thm:DPC} is a sufficient but not a necessary condition for convergence of $P(n)$ as $n \rightarrow \infty$.

Theorem \ref{thm:DPC} can be extended to Distributed Balancing Algorithm (DBA). In the DBA scheme, the iteration for the power levels is described by
\begin{equation}\label{eq:DBAPi}
P_i(n)=c(n-1)\Big(P_i(n-1)+\frac{P_i(n-1)}{\gamma_i(n-1)}\Big)
\end{equation}
for $i=1, \cdots, m$ and $n\geq 1$. We can write \eqref{eq:DBAPi} in matrix form as
\begin{equation}
P(n)=c(n-1)Z(n-1)P(n-1)
\end{equation}
where 
\[Z(n-1)=A(n-1)+I\]
\begin{thm}\label{thm:DPC}
For a system that uses the DBA scheme, with $c(n)$, $n\in \NN$ chosen such that $\displaystyle \lim_{n\rightarrow \infty}\displaystyle\prod_{i=0}^n c(n)$ is bounded,  then $\displaystyle \lim_{n\rightarrow \infty}P(n)$ is bounded if $\rho(\BB)<1$ where
\[\BB=\{B(1), \cdots, B(n)\}\]
with $B(i)=c(i)Z(i)$ and $Z(i)=A(i)+I$, for $i=1, \cdots, n$.
\end{thm}
\begin{proof}
The theorem can be proven using the same method as the proof of Theorem \ref{thm:DPC}.
\end{proof}
It should also be mentioned that calculating joint spectral radius for an arbitrary set of matrices is an NP hard problem as shown in \cite{TB97}. But the good news is that algorithms exist which can reach an arbitrary accuracy in an a priori computable amount of time. These algorithms use the concept of \emph{extremal norm} that is a norm with the property that $\|M\|\leq \MM$ for all $M\in \MM$.  For more information on these algorithms you can look at \cite{GWZ05,BNT05,BN05,PJ08,PJB10}. 

\section{Conclusions}\label{sec:conc}
In this manuscript, we presented a novel condition for convergence of transmitted power in a mobile network under Distributed Power Control (DPC) and Distributed Balancing Algorithm (DBA) schemes. We have used the concept of joint spectral radius to obtain the results.

\bibliographystyle{IEEEtrans}
\bibliography{The_Mother_of_All_Biblios_1304_2}

\end{document}